\documentclass[11pt,reqno,tbtags]{amsart}
\usepackage{amssymb}
\usepackage{url}
\usepackage[square,numbers]{natbib}
\bibpunct[, ]{[}{]}{;}{n}{,}{;}

{\makeatletter
\gdef\section{\@startsection{section}{1}%
  \z@{.7\linespacing\@plus\linespacing}{.5\linespacing}%
  {\normalfont\bfseries\centering}}
}

\title[Patterns in random permutations avoiding  multiple patterns]
{Patterns in random permutations avoiding some sets of multiple patterns}

\date{17 April, 2018}

\author{Svante Janson}
\thanks{Partly supported by the Knut and Alice Wallenberg Foundation}
\address{Department of Mathematics, Uppsala University, PO Box 480,
SE-751~06 Uppsala, Sweden}
\email{svante.janson@math.uu.se}
\newcommand\urladdrx[1]{{\urladdr{\def~{{\tiny$\sim$}}#1}}}
\urladdrx{http://www2.math.uu.se/~svante/}

\numberwithin{equation}{section}

\renewcommand\le{\leqslant}
\renewcommand\ge{\geqslant}

\allowdisplaybreaks

\newtheorem{theorem}{Theorem}[section]

\newtheorem{prop}[theorem]{Proposition}
\newtheorem{corollary}[theorem]{Corollary}

\theoremstyle{definition}
\newtheorem{example}[theorem]{Example}

\newtheorem{remark}[theorem]{Remark}

\theoremstyle{remark}

\newenvironment{romenumerate}[1][0pt]{
\addtolength{\leftmargini}{#1}\begin{enumerate}
 }{\end{enumerate}}

\newcounter{oldenumi}
{\setcounter{oldenumi}{\value{enumi}}
\begin{romenumerate} \setcounter{enumi}{\value{oldenumi}}}
{\end{romenumerate}}

\newcounter{thmenumerate}

\newcounter{romxenumerate}   

\newcounter{xenumerate}   

\newcommand{\refT}[1]{Theorem~\ref{#1}}

\newcommand{\refR}[1]{Remark~\ref{#1}}
\newcommand{\refS}[1]{Section~\ref{#1}}
\newcommand{\refSs}[1]{Sections~\ref{#1}}
\newcommand{\refSS}[1]{Subsection~\ref{#1}}
\newcommand{\refP}[1]{Proposition~\ref{#1}}
\newcommand{\refPs}[1]{Propositions~\ref{#1}}

\begingroup
  \divide\count255 by 60
  \multiply\count255 by -60
  \advance\count255 by \time
  \ifnum \count255 < 10 \xdef\klockan{\the\count1.0\the\count255}
  \else\xdef\klockan{\the\count1.\the\count255}\fi
\endgroup

\newcommand\nopf{\qed}   
\newcommand\qedtag{\eqno{\qed}}

\newcommand{\sumk}{\sum_{k=0}^\infty}

\newcommand{\sumin}{\sum_{i=1}^n}

\newcommand\set[1]{\ensuremath{\{#1\}}}
\newcommand\bigset[1]{\ensuremath{\bigl\{#1\bigr\}}}

\newcommand\bigpar[1]{\bigl(#1\bigr)}
\newcommand\Bigpar[1]{\Bigl(#1\Bigr)}

\newcommand\lrpar[1]{\left(#1\right)}

\newcommand\xcpar[1]{\{#1\}}

\def\rompar(#1){\textup(#1\textup)}    
\newcommand\xfrac[2]{#1/#2}

\newcommand\Bigparfrac[2]{\Bigpar{\frac{#1}{#2}}}

\def\xexp(#1){e^{#1}}

\newcommand\ntoo{\ensuremath{{n\to\infty}}}

\newcommand\xtoo{\ensuremath{{x\to\infty}}}

\newcommand\punkt{.\spacefactor=1000}    
\newcommand\iid{i.i.d\punkt}    
\newcommand\ie{i.e\punkt}
\newcommand\eg{e.g\punkt}

\newcommand\cf{cf\punkt}
\newcommand{\as}{a.s\punkt}

\newcommand{\tend}{\longrightarrow}
\newcommand\dto{\overset{\mathrm{d}}{\tend}}
\newcommand\pto{\overset{\mathrm{p}}{\tend}}

\newcommand\eqd{\overset{\mathrm{d}}{=}}

\newcommand\Op{O_{\mathrm p}}

\newcommand\bbR{\mathbb R}

\newcounter{CC}
\newcounter{cc}

\newcommand\E{\operatorname{\mathbb E{}}}
\renewcommand\P{\operatorname{\mathbb P{}}}
\newcommand\Var{\operatorname{Var}}
\newcommand\Cov{\operatorname{Cov}}

\newcommand\Bi{\operatorname{Bi}}

\newcommand\Be{\operatorname{Be}}
\newcommand\Ge{\operatorname{Ge}}

\newcommand\Dir{\operatorname{Dir}}

\newcommand\gam{\gamma}
\newcommand\gG{\Gamma}

\newcommand\gl{\lambda}
\newcommand\gL{\Lambda}

\newcommand\gs{\sigma}
\newcommand\gt{\tau}

\renewcommand\phi{\xxx}  

\newcommand\ett[1]{\boldsymbol1\xcpar{#1}}

\newcommand\qw{^{-1}}

\newcommand\qq{^{1/2}}
\newcommand\qqw{^{-1/2}}

\newcommand\intoi{\int_0^1}

\newcommand\oi{[0,1]}

\newcommand\setoi{\set{0,1}}

\newcommand\dd{\,\mathrm{d}}

\newcommand{\pgf}{probability generating function}

\newcommand\lhs{left-hand side}
\newcommand\rhs{right-hand side}

\newcommand\fS{\mathfrak S}
\newcommand\fSn{\fS_n}
\newcommand\fSm{\fS_m}
\newcommand\fSx{\fS_*}

\newcommand\fSxwww{\fSx(\www)}

\newcommand\fSxzzz{\fSx(\zzz)}
\newcommand\npp[1]{n_\perm{#1}}

\newcommand\nt{n_\gt}
\newcommand\ns{n_\gs}
\newcommand\ninv{n_\perm{21}}
\newcommand\perm[1]{\ensuremath{{#1}}}
\newcommand\pint{\pinx{\gt}}
\newcommand\pinT{\pinx{T}}
\newcommand\pinx[1]{\boldsymbol{\pi}_{{#1};n}}
\newcommand\pinxp[1]{\boldsymbol{\pi}_{\perm{#1};n}}

\newcommand\pix{\boldsymbol{\pi}}
\newcommand\pixn{\boldsymbol{\pi}_n}
\newcommand\pinwww{\pinx{\www}}
\newcommand\pinzzz{\pinx{\zzz}}
\newcommand\pinA{\pinx{\permA}}
\newcommand\pinB{\pinx{\permB}}

\newcommand\pinD{\pinx{\permD}}
\newcommand\pinE{\pinx{\permE}}
\newcommand\pinAAA{\pinx{\permAAA}}
\newcommand\pinBBB{\pinx{\permBBB}}
\newcommand\pinCCC{\pinx{\permCCC}}
\newcommand\pinEEE{\pinx{\permEEE}}

\newcommand\permA{\perm{231}, \perm{321}}
\newcommand\permB{\perm{231}, \perm{312}}
\newcommand\permD{\perm{132}, \perm{312}}
\newcommand\permE{\perm{132}, \perm{321}}

\newcommand\permAAA{\perm{231},\perm{312}, \perm{321}}
\newcommand\permBBB{\perm{132},\perm{231}, \perm{321}}
\newcommand\permCCC{\perm{132},\perm{231}, \perm{312}}
\newcommand\permEEE{\perm{132},\perm{213}, \perm{321}}

\newcommand\nx[1]{n_{\perm{#1}}}
\newcommand\ww{\perm{21}}
\newcommand\www{\perm{321}}
\newcommand\zzz{\perm{132}}

\newcommand\bex{\mathbf{e}}

\newcommand\qr{^{\mathsf r}}
\newcommand\qc{^{\mathsf c}}
\newcommand\qs{^{\mathsf s}}
\newcommand\qsx{{\mathsf s}}
\newcommand\iotax{\bar\iota}
\newcommand\fG{\mathfrak G}
\newcommand\Ustat{$U$-statistic}
\newcommand\ggs{\beta}
\newcommand\ggss{\ggs^2}
\newcommand\xix{\xi'}
\newcommand\NN[1]{N_{#1}}
\newcommand\gamx{\gamma}
\newcommand\gamxx{\gamx^2}
\newcommand\Uoi{\mathsf U(0,1)}

\begin{document}

\begin{abstract} 
We consider a random permutation drawn from
the set of  permutations of length $n$ 
that avoid some given set of patterns of length 3.
We show that
the number of occurrences of another pattern $\sigma$ has a limit distribution,
after suitable scaling.
In several cases, the number is asymptotically normal; this contrasts to the
cases of permutations avoiding a single pattern of length 3 studied  in
earlier papers.
\end{abstract}

\maketitle

\section{Introduction}\label{S:intro}

Let $\fS_n$ be the set of permutations of $[n]:=\set{1,\dots,n}$,
and $\fSx:=\bigcup_{n\ge1}\fS_n$.
If $\gs=\gs_1\dotsm\gs_m\in\fS_m$ and $\pi=\pi_1\dotsm\pi_n\in\fS_n$,
then an \emph{occurrence} of $\gs$ in $\pi$ is a 
subsequence $\pi_{i_1}\dotsm\pi_{i_m}$, 
with $1\le i_1<\dots<i_m\le n$, 
that
has the same order as
$\gs$, i.e., 
$\pi_{i_j}<\pi_{i_k} \iff \gs_j<\gs_k$ for all $j,k\in [m]$.
We let $\ns(\pi)$ be the number of occurrences of $\gs$ in $\pi$, and note
that
\begin{equation}\label{11}
  \sum_{\gs\in\fS_m} \ns(\pi) = \binom nm,
\end{equation}
for every $\pi\in\fS_n$.
For example, an inversion is an occurrence of \perm{21}, and thus
$\npp{21}(\pi)$ is the number of inversions in $\pi$.

We say that $\pi$ \emph{avoids} another permutation 
$\tau$ if $\nt(\pi)=0$; otherwise, $\pi$
\emph{contains} $\tau$. Let 
\begin{equation}\label{fsgt}
\fS_n(\tau):= \set{\pi\in\fS_n:\nt(\pi)=0},
\end{equation}
the set 
of permutations of length $n$ that avoid $\tau$.
More generally, for any  set $T=\set{\tau_1,\dots,\tau_k}$ of permutations,
let
\begin{equation}
  \label{fsgtt}
  \fS_n(T)=\fS_n(\tau_1,\dots,\tau_k):=\bigcap_{i=1}^k\fS_n(\tau_i),
\end{equation}
the set of permutations of length $n$ that avoid all $\tau_i\in T$.
We also let $\fSx(T):=\bigcup_{n=1}^\infty\fS_n(T)$ be the set of
$T$-avoiding permutations of arbitrary length. 

The classes $\fSx(\gt)$ and, more generally, 
$\fSx(T)$ 
have been studied
for a long time, see \eg{}
\citet[Exercise 2.2.1-5]{KnuthI}, 
\citet{SS}, 
\citet{Bona}.
In particular, 
one classical  problem is to enumerate the sets
$\fS_n(\gt)$, either exactly or asymptotically,
see \citet[Chapters 4--5]{Bona}.
We note the fact that for any $\tau$ with length $|\tau|=3$,
$\fS_n(\gt)$ has the same size 
$
|\fS_n(\gt)|=C_n:=\frac{1}{n+1}\binom{2n}{n}
$,
the $n$-th Catalan number, see \eg{} 
\cite[Exercises 2.2.1-4,5]{KnuthI}, 
\cite{SS},
\cite[Exercise 6.19ee,ff]{StanleyII},
\cite[Corollary 4.7]{Bona}; furthermore, 
the cases when $T$ consists of several permutations of length 3
are all treated by \citet{SS}.
(The situation for $|\tau|\ge4$ is more complicated.)

The general problem that concerns us is to
take a fixed set $T$ of one or several  permutations 
and let $\pinT$ be a uniformly random $T$-avoiding permutation, \ie, a
uniformly random element of $\fS_n(T)$, and then study the 
distribution of the random variable $\ns(\pinT)$
for some other fixed permutation $\gs$.
(Only $\gs$ that are themselves $T$-avoiding are interesting, 
since otherwise $\ns(\pinT)=0$.)
One instance of this problem was studied already by \citet{RobertsonWZ},
who gave a  generating function for $\npp{123}(\pinxp{132})$.
The exact distribution of $\ns(\pint)$ 
for a given $n$ was studied 
numerically
in \cite{SJ287}, where higher
moments and mixed moments
are calculated for small $n$. 
We are mainly interested in 
asymptotics of the distribution 
of $\ns(\pinT)$, and of its moments,
as \ntoo{}, for some fixed $T$ and $\gs$. 

In the present paper we study the cases when $T$ is a set of two or more
permutations of length 3. The cases when $T=\set\tau$ for a single
permutation $\tau$ of length $|\tau|=3$ 
were studied in \cite{SJ290,SJ324} (by symmetries, see \refS{SSsymm}, 
only two such cases have to be considered), and the cases when $T$ contains a
permutation of length $\le 2$ are trivial (there is then at most one permutation
in $\fS_n(T)$ for any $n$); hence the present paper completes the study of
forbidding one or several permutations of length $\le 3$.
The case of forbidding one or several permutations of length $\ge4$ seems much
more complicated, but there are recent impressive results in some cases by
\citet{Bassino1}
and \citet{Bassino2}.

The 
expectation $\E\ns(\pinT)$, or
equivalently, the total number of 
occurences of $\gs$ in all $T$-avoiding permutations, has 
previously been treated
in a number of papers
for various cases,
beginning with \citet{Bona-abscence,Bona-surprising} (with $\tau=\perm{132}$).
In particular, 
\citet{Zhao} has given exact formulas 
when $|\gs|=3$
for the (non-trivial) cases 
treated in the present paper, where $T$ consist of two or more permutations
of length 3.

\begin{remark}\label{Rnormal}
For
the non-restricted case of uniformly random permutations in $\fS_n$,
it is well-known that if $\pixn$ is a uniformly random permutation in $\fS_n$,
then $\ns(\pixn)$ has an asymptotic normal distribution as \ntoo{}
for every fixed permutation $\gs$; more precisely, if $|\gs|=m$ then,
as \ntoo,
\begin{equation}\label{null}
  \frac{\ns(\pixn)-\frac{1}{m!}\binom nm}{n^{m-1/2}}\dto N\bigpar{0,\gam^2}
\end{equation}
for some $\gam^2>0$ depending on $\gs$;
see \citet{Bona-Normal,Bona3}
and \citet[Theorem 4.1]{SJ287}.

We obtain below similar asymptotic normal results in several cases
(\refSs{SD}, \ref{SB}, \ref{SA}, \ref{S3a}); note that the asymptotic
normality in particular implies concentration in these cases, in the sense
\begin{equation}
\frac{  \ns(\pinT)}{\E \ns(\pinT)}\pto 1.
\end{equation}
On the other hand, in other cases 
(\refSs{S1}, \ref{SE}, \ref{S3c}, \ref{S3b}, \ref{S3e})
we find a different type of limits,
where 
$\xfrac{  \ns(\pinT)}{\E \ns(\pinT)}$ converges to some non-trivial positive
random variable. The same holds in the case $T=\set{2413,3142}$ studied by  
\citet{Bassino1}.

We see no obvious pattern in the occurence of these two types of limits in
the cases below; nor do we know whether these are the only possibilities for
a general set $T$ of forbidden permutations.
\end{remark}

\begin{remark}
  In the present paper we consider  for simplicity often
only univariate limits; corresponding multivariate   results for several 
$\gs_1,\dots,\gs_k$ follow by the same methods. In particular, \eqref{null}
and
all instances
of normal limit laws below  extend to multivariate normal limits,
with covariance matrices that can be computed explicitly.
\end{remark}
  
\begin{remark}
In the present paper we study only the numbers $\ns$ of occurences of some
pattern in $\pint$.
There is also a number of papers by various authors
that study other properties of 
random $\tau$-avoiding permutations, see \eg{} the references in \cite{SJ324};
such results will not be considered here.

\end{remark}

\section{Preliminaries}

\subsection{Notation}

Let $\iota=\iota_n$ be the identity permutation of length $n$.
Let $\iotax_n=n\dotsm21$ be its reversal.

Let $\pi=\pi_1\dotsm\pi_n$ be a permutation.
We say that a value $\pi_i$ is a \emph{maximum}
if $\pi_i>\pi_j$ for every $j<i$, 
and a \emph{minimum} if $\pi_i<\pi_j$ for every $j<i$.
(These are sometimes called \emph{LR maximum} and \emph{LR minimum}.)
Note that $\pi_1$ always is both a maximum and a minimum.

\subsection{Symmetries}\label{SSsymm}
There are many cases treated in the present paper, but the number is
considerably reduced by three natural symmetries (used by many previous
authors). 
For any permutation $\pi=\pi_1\dotsm\pi_n$, define its
\emph{inverse} $\pi\qw$ in the usual way, and its \emph{reversal} and
\emph{complement} by
\begin{align}
\pi\qr&:=\pi\circ\iotax =\pi_n\dotsm\pi_1,  
\\
\pi\qc&:=\iotax\circ\pi =(n+1-\pi_1)\dotsm(n+1-\pi_n).
\end{align}
These three operations are all involutions, and they generate a group 
$\fG$
of 8 symmetries (isomorphic to the dihedral group $D_4$).
It is easy to see that, for any permutations $\gs$ and $\pi$,
\begin{equation}\label{symm1}
  n_{\gs\qw}(\pi\qw)=  n_{\gs\qr}(\pi\qr) = n_{\gs\qc}(\pi\qc) =   n_{\gs}(\pi),
\end{equation}
and 
consequently, for any
symmetry $\qsx\in\fG$, 
\begin{equation}\label{symm}
  n_{\gs\qs}(\pi\qs) =   n_{\gs}(\pi).
\end{equation}
For a set $T$ of permutations we define $T\qs:=\set{\tau\qs:\tau\in T}$.
It follows from \eqref{symm} that
\begin{equation}
  \fS_n(T\qs)=\set{\pi\qs:\pi\in\fS_n(T)},
\end{equation}
and, furthermore, that for any permutation $\gs$,
\begin{equation}\label{symmT}
  n_{\gs\qs}(\pinx{T\qs})\eqd \ns(\pinT).
\end{equation}
We say that the sets of forbidden permutations $T$ and $T\qs$ are
\emph{equivalent}, 
and note that \eqref{symmT} implies that
it suffices to consider one set $T$ in each equivalence class 
$\set{T\qs:\qsx\in\fG}$. 
We do this in the sequel without further comment.
(We choose representatives $T$ that we find convenient. One guide is that we
choose $T$ such that the identity permutation $\iota_n$ avoids $T$.)

\subsection{Compositions and decompositions of permutations}
\label{SSblocks}
If $\gs\in\fS_m$ and $\tau\in\fS_n$,  their \emph{composition}
$\gs*\tau\in\fS_{m+n}$ is defined
by letting $\tau$ act on $[m+1,m+n]$ in the natural
way; more formally,
$\gs*\tau=\pi\in\fS_{m+n}$ where $\pi_i=\gs_i$ for $1\le i\le m$, and
$\pi_{j+m}=\tau_j+m$ for $1\le j\le n$. 
It is easily seen that $*$ is an associative operation that makes $\fS_*$
into a semigroup (without unit, since we only consider permutations of
length $\ge1$). We say that a permutation $\pi\in\fS_*$  is
\emph{decomposable} if $\pi=\gs*\tau$ for some $\gs,\tau\in\fS_*$, and
\emph{indecomposable} otherwise;
we also call an indecomposable permutation a
\emph{block}.
Equivalently, $\pi\in\fS_n$ is decomposable if and only if $\pi:[m]\to[m]$
for some $1\le m<n$. 
See \eg{}  \cite[Exercise VI.14]{Comtet}.

It is easy to see that any permutation $\pi\in\fS_*$ has a unique
decomposition $\pi=\pi_1*\dots*\pi_\ell$  into indecomposable permutations
(blocks)  $\pi_1,\dots,\pi_\ell$ (for some, unique, $\ell\ge1$); 
we call these the \emph{blocks of $\pi$}.

We shall see that some (but not all) of the classes considered below can be
characterized in terms of their blocks. (See \cite{Bassino1} for another,
more complicated, example.)

\subsection{\Ustat{s}}\label{SSU}

An (asymmetric) \emph{\Ustat} is a random variable of the form
\begin{equation}\label{U}
  U_n = \sum_{1\le i_1<\dots<i_d\le n} f\bigpar{X_{i_1},\dots,X_{i_d}},
\qquad n\ge0,
\end{equation}
where $X_1,X_2,\dots$ is an \iid{} sequence of random variables and $f$ is a
given function of $d\ge1$ variables.
These were (in the symmetric case) introduced by \citet{Hoeffding};
see further \eg{} \cite{SJ332} and the references there.
We say that $d$ is the \emph{order} of the \Ustat.

We shall use the central limit theorem for \Ustat{s}, originally due to
\citet{Hoeffding}, in the asymmetric version given in 
\cite[Theorem 11.20]{SJIII} and 
\cite[Corollary 3.5 and (moment convergence) Theorem 3.15]{SJ332}. 
Let, with $X$ denoting a
generic $X_i$,
\begin{align}
  \mu&:=\E f(X_1,\dots,X_d), \label{Umu}
\\
f_i(x)&:=\E\bigpar{f(X_1,\dots,X_d)\mid X_i=x}, 
\\
\ggs_{ij}&:=\Cov \bigpar{f_i(X),f_j(X)},
\\\label{Ugss}
      \ggss&:=
\sum_{i,j=1}^d
\frac{(i+j-2)!\,(2d-i-j)!}{(i-1)!\,(j-1)!\,(d-i)!\,(d-j)!\,(2d-1)!}\ggs_{ij}
.
\end{align}
Note that $f_i(x)$ in \cite{SJ332} is  $f_i(x)-\mu$ in the present notation.

\begin{prop}[\cite{SJIII,SJ332}]\label{PU}
Suppose that $f(X_1,\dots,X_d)\in L^2$.
Then, 
with the notation in \eqref{Umu}--\eqref{Ugss},
as \ntoo,
\begin{equation}\label{c1}
  \frac{U_n-\binom nd \mu}{n^{d-1/2}} \dto N\bigpar{0,\ggss}.
\end{equation}
Furthermore, $\ggss>0$ unless $f_i(X)=\mu$
\as{} for $i=1,\dots,d$.

Moreover, if $f(X_1,\dots,X_d)\in L^p$ for some $p\ge2$, 
the \eqref{c1} holds with convergence of all moments of order $\le p$.
\qed
\end{prop}

\begin{example}\label{Eunrestricted}
  A uniformly random permutation 
$\pix_n$ of length $n$  (without other restrictions) can be constructed as
the relative order of $X_1,\dots,X_n$, where $X_i$ are \iid{} with, for
example, a uniform distribution $\Uoi$.
For any given permutation $\gs\in\fS_m$, we can then write $\ns(\pix_n)$ as
a \Ustat{} \eqref{U} for a suitable indicator function $f$.
Then \refP{PU} yields a limit theorem showing that $\ns(\pix_n)$ is
asymptotically normal. See \cite{SJ287} for details. 
\end{example}

We shall also use a renewal theory version of \refP{PU}.
With the notations above,
assume (for simplicity) that $X_i\ge0$.
Define $S_n:=\sumin X_i$, and let for each $x>0$
\begin{align}
\NN-(x)&:=\sup\set{n:S_n< x},\label{NN-}
\\ 
\NN+(x)&:=\inf\set{n:S_n\ge x}=\NN-(x)+1. \label{NN+}
\end{align}
\begin{remark}\label{RNN}
  The definitions \eqref{NN-}--\eqref{NN+} differ slightly from the ones in
  \cite{SJ332}, where instead $S_n\le x$ and $S_n>x$ are used. This does not
  affect the asymptotic results used here.
Note that the event $\set{S_k=n \text{ for some }k\ge0}$ equals
\set{S_{\NN+(n)}=n} in the present notation.
\end{remark}

 The following results are special cases of
\cite[Theorems  3.11, 3.13(iii) and 3.18]{SJ332} (with somewhat different
notation). 
$\NN\pm(x)$ means either $\NN-(x)$ or $\NN+(x)$; the results holds for both.

\begin{prop}[\cite{SJ332}]\label{PUN}
\newcommand{\td}{d}
\newcommand{\tf}{f}
\newcommand{\tU}{U}
\newcommand{\tmu}{\mu}
\newcommand{\mux}{\nu}
Suppose that $f(X_1,\dots,X_d)\in L^2$, $X\in L^2$, $X\ge0$ \as, and
$\mux:=\E X>0$.
Then, with notations as above, 
as \xtoo,
\begin{equation}\label{cvtau}
\frac{\tU_{\NN\pm(x)}-{\mux}^{-\td}{\tmu}{\td!}\qw x^{\td}}
{x^{\td-1/2}} 
\dto N\bigpar{0,\gamx^2},
\end{equation}
where
\begin{align}\label{cvtau2}
    \gamx^2
&:=
{\mux}^{1-2\td}
\ggss
-2\frac{{\mux}^{-2\td}\tmu}{(\td-1)!\,\td!}\sum_{i=1}^{\td}
\Cov\bigpar{\tf_i(X),X}
+ \frac{{\mux}^{-2\td-1}\tmu^2}{(\td-1)!^2}\Var\bigpar{X}.
\end{align}
Moreover,
$\gamx^2>0$ unless  $f_i(X)=\frac{\mu}{\nu} X$ \as{} for $i=1,\dots,d$.
\qed  
\end{prop}

\begin{prop}[\cite{SJ332}] \label{PUN2}
Suppose in addition to the hypotheses in \refP{PUN} that $X$ is
integer-valued.
Then \eqref{cvtau} holds also conditioned on $S_{\NN+(x)}=x$ 
(\cf{} \refR{RNN})
for integers $x\to\infty$.
\qed
\end{prop}

\begin{prop}[\cite{SJ332}] \label{PUNmom}
Suppose in addition to the hypotheses in \refP{PUN} or \ref{PUN2} 
that 
$f(X_1,\dots,X_d)\in L^p$ and  $X\in L^p$ for every $p<\infty$.
Then the conclusion \eqref{cvtau} holds 
with convergence of all moments.
\qed
\end{prop}

\subsection{Trivial cases}\label{S0}

We consider in the present paper
sets $T\subseteq\fS_3$. Then, see \cite{SS}, the following cases
are trivial in the sense that for all $n\ge5$, $|\fSn(T)|=0$, 1 or 2.
\begin{romenumerate}
\item 
$T=\set{123,321}$,
\item 
$|T|=3$ and $T\supset\set{123,321}$,
\item 
$|T|\ge4$.
\end{romenumerate}
We ignore these cases in the sequel. This leaves $6$ cases with $|T|=1$
(\refS{S1}),
14 cases with $|T|=2$ (\refSs{SD}--\ref{SE}),
and 16 cases with $|T|=3$
(\refSs{S3a}--\ref{S3e}). Symmetries reduce these to the $2+4+4=10$
non-equivalent cases discussed below.

\section{Avoiding a single permutation of length 3}\label{S1}

There are 6 cases where a single permutation of length 3 is avoided,
but by the symmetries in \refSS{SSsymm} these reduce to 2 non-equivalent
cases, for example 132  (equivalent to 231, 213, 312)
and 321 (equivalent to 123).
These cases are treated in detail in \cite{SJ290} and \cite{SJ324},
respectively.
Both analyses are based on bijections with binary trees and Dyck paths, and
the well-known convergence in distribution of random Dyck paths to a
Brownian excursion, but the details are very different, and so are in
general the resulting limit distributions.

For comparison with the results in later sections, 
we quote the main results of \cite{SJ290} and \cite{SJ324},
referring to these papers for further details and proofs.
Recall that the standard Brownian excursion $\bex(x)$ is a random
non-negative function on $\oi$.

First, for 132, let
\begin{equation}\label{gl}
  \gl(\gs):=|\gs|+D(\gs)
\end{equation}
where $D(\gs)$ is the number of 
\emph{descents} in $\gs$, \ie, indices
$i$ such that $\gs_i>\gs_{i+1}$ or (as a convenient convention) $i=|\gs|$.
Note that  $1\le D(\gs)\le|\gs|$, and thus
\begin{equation}
|\gs|+1\le\gl(\gs)\le 2|\gs|,  
\end{equation}
with the extreme values $\gl(\gs)=|\gs|+1$ if and only if $\gs=1\dotsm k$,
and
$\gl(\gs)=2|\gs|$ if and only if $\gs=k\dotsm 1$, for some $k=|\gs|$.

\begin{theorem}[\cite{SJ290}]\label{T132}
There exist 
strictly positive random variables $\gL_\gs$
such that
\begin{equation}\label{tmaind}
 \ns(\pinzzz)/ n^{\gl(\gs)/2} \dto \gL_\gs,
\end{equation}
as \ntoo, jointly for all $\gs\in\fSxzzz$. 
Moreover, this holds with convergence of
all moments. 

For a monotone decreasing permutation $k\dotsm1$,  
$\gL_{k\dotsm 1}=1/k!$ is deterministic, but not for any other $\gs$.
\nopf
\end{theorem}

The limit variables $\gL_\gs$ in \refT{T132} can be expressed as functionals
of a
Brownian excursion $\bex(x)$, see \cite{SJ290};
the description is, in general, rather complicated, but some cases are simple.

\begin{example}
In the special case $\gs=12$, $\gL_{12}=\sqrt2\intoi\bex(x)\dd x$,
see \cite[Example 7.6]{SJ290}; 
this is
(apart from the factor $\sqrt2$) 
the well-known \emph{Brownian excursion area},
see \eg{} \cite{SJ201} and the references there. 

For the number $n_{21}$ of inversions, we thus have 
\begin{equation}\label{qi}
  \frac{\binom n2 - n_{21}(\pinzzz)}{n^{3/2}}
=
  \frac{n_{12}(\pinzzz)}{n^{3/2}}
\dto \gL_{12}
=\sqrt2\intoi\bex(x)\dd x.
\end{equation}
By \refSS{SSsymm}, 
the \lhs{} can also be seen as the number of inversions
$\ninv(\pinx{231})$ or $\ninv(\pinx{312})$, 
normalized by $n^{3/2}$, 
where we instead avoid 231 or 312.
\end{example}

\begin{theorem}[\cite{SJ324}]\label{T1}
  Let $\gs\in\fSxwww$.  
Let $m:=|\gs|$, and
suppose that $\gs$ has $\ell$ blocks of lengths $m_1,\dots,m_\ell$.
Then, as \ntoo,
\begin{equation}\label{t1a}
  \ns(\pinx{\www})/n^{(m+\ell)/2}\dto W_\gs
\end{equation}
for a positive random variable $W_\gs$ that can be represented as
\begin{equation}\label{t1w}
  W_\gs= w_\gs\int_{0<t_1<\dots<t_\ell<1} 
\bex(t_1)^{m_1-1}\dotsm \bex(t_\ell)^{m_\ell-1}\dd t_1\dotsm \dd t_\ell,
\end{equation}
where
$w_\gs$ is positive constant. 

Moreover, the convergence \eqref{t1a} holds jointly for any set of
$\gs\in\fSxwww$, and with convergence of all moments.
\end{theorem}

\begin{example}\label{E21}
  Let $\gs=\perm{21}$. Then $w_{21}=2\qqw$, see \cite{SJ324},
and thus  \eqref{t1a}--\eqref{t1w}, with $\ell=1$ and $m_1=m=2$, 
yield 
for the number of inversions,
\begin{equation}\label{qj}
 \frac{ n_{21}(\pinwww)}{ n^{3/2}}\dto 2\qqw \intoi \bex(x)\dd x.
\end{equation}
Note that the limit in \eqref{qj} differs from the one in \eqref{qi} by a
factor 2.
\end{example}

\section{Avoiding \set{\perm{132}, \perm{312}}}\label{SD}

In this section we avoid $T=\set{\permD}$. 
Equivalent sets are 
\set{132,231}, 
\set{213, 231},  
\set{213, 312}.  

It was shown by \citet{SS} that $|\fSn(\permD)|=2^{n-1}$, together with the
following characterization (in an equivalent formulation).

\begin{prop}[{\cite[Proposition 12]{SS}}]\label{PD}
A permutation $\pi$ belongs to the class
 $\fSx(\permD)$ if and only if every 
entry $\pi_i$ is either a maximum or a minimum.
\qed
\end{prop}

We encode a permutation $\pi\in\fSn(\permD)$ by a sequence
$\xi_2,\dots,\xi_n\in\set{\pm1}^{n-1}$, where $\xi_j=1$ if $\pi_j$ is a
maximum in $\pi$, and $\xi_j=-1$ if $\pi_j$ is a minimum.
This is by \refP{PD} a bijection, and hence the code for a
uniformly random $\pinD$ has $\xi_2,\dots,\xi_n$ \iid{} with
the symmetric Bernoulli distribution
$\P(\xi_j=1)=\P(\xi_j=-1)=\frac12$.
We let $\xi_1$ have the same distribution and be independent of
$\xi_2,\dots,\xi_n$. 

Let  $\gs\in\fSm(\permD)$ have the code $\eta_2,\dots,\eta_m$.
Then $\pi_{i_1}\dotsm\pi_{i_m}$ is an occurrence of $\gs$ in $\pi$
if and only if $\xi_{i_j}=\eta_j$ for $2\le j\le m$.
Consequently, \cf{} \refSS{SSU},
$\ns(\pinD)$ is a \Ustat{}
\begin{equation}\label{nsD}
  \ns(\pinD)=\sum_{i_1<\dots<i_m}f\bigpar{\xi_{i_1},\dots,\xi_{i_m}},
\end{equation}
where
\begin{equation}\label{fD}
  f\bigpar{\xi_1,\dots,\xi_m}:=\prod_{j=2}^m\ett{\xi_j=\eta_j}.
\end{equation}
Note that $f$ does not depend on the first argument.
It follows that, with the notation \eqref{Umu}--\eqref{Ugss}, 
\begin{align}
 \mu &= \E f\bigpar{\xi_1,\dots,\xi_m}=2^{-(m-1)},
\label{muD}
\\ 
  f_i(\xi)&=
  \begin{cases}
0, & i=1,
\\
2^{-(m-2)}\ett{\xi=\eta_i},
& 2\le i\le m,
  \end{cases}
\label{fiD}
\\
  \ggs_{ij}&=\Cov \bigpar{f_i(\xi),f_j(\xi)} = 2^{2-2m}\eta_i\eta_j, 
\qquad i,j\ge2,
\label{gsijD}
\\
  \ggss&=
2^{2-2m}
\sum_{i,j=2}^m
\frac{(i+j-2)!\,(2m-i-j)!}{(i-1)!\,(j-1)!\,(m-i)!\,(m-j)!\,(2m-1)!}\eta_i\eta_j
.
\label{gssD}
\end{align}
\refP{PU} yields:
\begin{theorem}\label{TD}
  For any $m\ge1$ and $\gs\in\fSm(\permD)$, as \ntoo,
  \begin{equation}\label{td}
    \frac{\ns(\pinD)-2^{1-m} n^{m}/m!}{n^{m-1/2}}
\dto N\bigpar{0,\ggss},
  \end{equation}
with $\ggss>0$ given by \eqref{gssD}.

Moreover, \eqref{td} holds with convergence of all moments.
\end{theorem}


\begin{example}
  For the number of inversions, we have $\gs=21$ and $m=2$, $\eta_2=-1$.
Thus, \eqref{fiD} yields $f_1(\xi)=0$ and $f_2(\xi)=\ett{\xi=-1}$.
We find, from \eqref{muD}--\eqref{gssD}, 
$\mu=\frac12$, $\ggs_{22}=\frac{1}4$ and 
$\ggss=\frac{1}{12}$, and thus \refT{TD} yields
  \begin{equation}
    \frac{\ninv(\pinD)- n^{2}/4}{n^{3/2}}
\dto N\bigpar{0,\tfrac1{12}},
  \end{equation}
\end{example}

\begin{remark}\label{RD}
It is easily
  seen from \eqref{nsD}--\eqref{fD} that the expected number of occurrences
$\E\ns(\pinD) =2^{1-m}\binom nm$, 
for every $\gs\in\fSm(\permD)$; hence the expectation depends only on the
length $m=|\gs|$.

The variance depends not only on $|\gs|$, not even asymptotically, by
\eqref{gssD}.
\end{remark}

\section{Avoiding \set{\perm{231}, \perm{312}}}\label{SB}

In this section we consider $T=\set{\permB}$. 
The only equivalent set is
\set{132, 213}.    

It was shown by \citet{SS} that $|\fSn(\permB)|=2^{n-1}$, together with the
following characterization (in an equivalent form).

\begin{prop}[{\cite[Proposition 12]{SS}}]\label{PB}
A permutation $\pi$ belongs to the class 
$\fSx(\permB)$ if and only if every block in
$\pi$ is decreasing, i.e., of the type $\ell(\ell-1)\dotsm 21$ for some $\ell$. 
\qed
\end{prop}

Hence there exists exactly one block of each length $\ell\ge1$, 
and a permutation $\pi\in\fSx(\permB)$ is uniquely determined by the block
lengths.
In this section, let $\pi_{\ell_1,\dots,\ell_b}$ denote the  permutation in
$\fSx(\permB)$ with block lengths $\ell_1,\dots,\ell_b$, i.e.,
\begin{equation}\label{piB}
  \pi_{\ell_1,\dots,\ell_b}:=\iotax_{\ell_1}*\dotsm *\iotax_{\ell_b}.
\end{equation}

If $\gs,\pi\in\fSx(\permB)$, then in an occurrence of $\gs$ in $\pi$, each
block in $\gs$ has to be mapped into a block in $\pi$, and distinct blocks
have to be mapped into distinct blocks.
Conversely, any such increasing map $[m]\to[n]$ defines an occurence of
$\gs$.
It follows that if $\gs=\pi_{\ell_1,\dots,\ell_b}$, then
\begin{equation}\label{nsB}
  \ns\bigpar{\pi_{L_1,\dots,L_B}} 
= \sum_{1\le i_1<\dots<i_b\le B} \prod_{j=1}^b \binom{L_{i_j}}{\ell_i}.
\end{equation}
This is similar to a \Ustat{} \eqref{U}, but note that if we write $\pinB$ as
$\pi_{L_1,\dots,L_B}$, then the block lengths $L_1,\dots,L_B$ are not
independent (since their sum is fixed $=n$), and the number of blocks $B$ is
random. However, we can analyze this variable using the renewal theory in
\refSS{SSU} as follows.

First, mark each endpoint of the blocks in $\pi\in\fSn(\permB)$ by 1, and
mark all 
other indices in $[n]$ by $0$. Thus $\pi$ defines a string
$\xi_1,\dots,\xi_n\in\setoi^n$,
where necessarily $\xi_n=1$ but $\xi_1,\dots,\xi_{n-1}$ are arbitrary.
This yields a bijection between $\fSn(\permB)$ and the $2^{n-1}$ such
strings; hence, we obtain a uniformly random $\pinB$ by letting
$\xi_1,\dots,\xi_{n-1}$ be \iid{} $\Be(\frac12)$, \ie, with
$\P(\xi_i=0)=\P(\xi_i=1)=\frac12$. 

We change notation a little, to avoid problems at the endpoint, and define
$\xix_1,\xix_2,\dots$ as an infinite \iid{} sequence with
$\xix_i\sim\Be(\frac12)$.
Regard each $i$ with $\xix_i=1$ as the end of a block, and 
let $X_1,X_2,\dots$, be the successive  lengths of these (infinitely
many) blocks. Then $X_i$ are \iid{} with 
\begin{equation}\label{Ge}
X_i\sim\Ge(\tfrac12).  
\end{equation}
Given $n$, we then may let $\xi_i:=\xix_i$ for $1\le i<n$, and $\xi_n:=1$;
this determines $\xi_1,\dots,\xi_n$ as above, and thus a uniformly random
$\pinB$.
With this construction, the number of blocks in $\pinB$ is, 
recalling \eqref{NN-}--\eqref{NN+}, 
$B=\NN+(n)$, and the block lengths are
\begin{equation}\label{GeL}
  L_i=
  \begin{cases}
    X_i,& i<\NN+(n)
\\
n-\sum_{i<\NN+(n)}X_i \le X_{\NN+(n)}, & i=\NN+(n).
  \end{cases}
\end{equation}
Consequently,
if $\gs=\pi_{\ell_1,\dots,\ell_b}$ and we define
\begin{equation}\label{fB}
f(x_1,\dots,x_b):=\prod_{j-1}^b\binom{x_i}{\ell_i},    
\end{equation}
then \eqref{nsB} and \eqref{U} show that
\begin{equation}\label{UB}
U_{\NN-(n)}\le  \ns(\pinB) \le U_{\NN+(n)}.
\end{equation}
Consequently, the asymptotic result in \eqref{cvtau}, which holds for both
$U_{\NN-(n)}$  and $U_{\NN+(n)}$, holds also for $\ns(\pinB)$. 

\begin{remark}\label{Ralt}
  Alternatively, we can obtain $(\xi_i)$ from $(\xi'_i)$ by conditioning on
$\xi'_n=1$, and note that this holds when $S_{\NN+(n)}=n$ (see \refR{RNN}), 
and then 
$\ns(\pinB) = U_{\NN+(n)}$. The result then follows from \refP{PUN2}.
\end{remark}

To calculate the parameters, note that,
by \eqref{Ge}, $X$ has the \pgf{}
\begin{equation}\label{pgfB}
 g(z):= \E z^X = \sumk 2^{-k}z^k=\frac{z}{2-z}=\frac{2}{2-z}-1
\end{equation}
and it follows that for any integers $k,l\ge0$ with $(k,l)\neq(0,0)$,
\begin{align}\label{essB}
\E\lrpar{ \binom{X}{k}\binom{X}{\ell}}
&=[z^kw^\ell]\E \bigpar{(1+z)^X(1+w)^X}
\\&
=[z^kw^\ell]g\bigpar{(1+z)(1+w)}
\\&
=[z^kw^\ell]\frac{2}{2-(1+z)(1+w)}
\\&
=[z^kw^\ell]\frac{2}{1-z-w-zw}
\\&
=2D(k,\ell)
=2\sum_{i=0}^{k\land\ell}\frac{(k+\ell-i)!}{(k-i)!\,(\ell-i)!\,i!}
\end{align}
where $D(k,\ell)$ denotes the Delannoy numbers.
($D(k,\ell)$ is, e.g., the number of lattice paths from $(0,0)$ to $(k,\ell)$
with steps $(1,0)$, $(0,1)$ or $(1,1)$; see \cite[Example 6.3.8]{StanleyII}
and \cite[A008288 and A001850]{OEIS} and the references there.)
Simple
calculations
then yield
\begin{align}\label{Balpha}
\nu&=\E X = 2,  
\\
\Var(X)&=2.  \label{BvarX}
\\
\mu&=\E \prod_{j=1}^b\binom{X}{\ell_i} 
= \prod_{j=1}^b\E\binom{X}{\ell_i} 
= 2^b,
\\
f_i(X)&=2^{b-1}\binom{X}{\ell_i}, 
\\
\ggs_{ij}&=\Cov\bigpar{f_i(X),f_j(X)}
=2^{2b-1}D(\ell_i,\ell_j)-2^{2b},
\\
\Cov\bigpar{f_i(X),X}&=2^bD(\ell_i,1)-2^{b+1}=(2\ell_i-1)2^b
.  \label{Bomega}
\end{align}
Consequently, we obtain by \refPs{PUN} and \ref{PUNmom}
asymptotic normality in the following form.

\begin{theorem}\label{TB}
Let $\gs\in\fSm(\permB)$ have block lengths $\ell_1,\dots,\ell_b$. Then,
as \ntoo,
\begin{equation}\label{tb}
  \frac{\ns(\pinB)-n^b/b!}{n^{b-1/2}}
\dto N\bigpar{0,\gamxx},
\end{equation}
where $\gamxx$ can be calculated by \eqref{cvtau2} and
\eqref{Balpha}--\eqref{Bomega}.

Moreover, \eqref{tb} holds with convergence of all moments.
\qed
\end{theorem}

\begin{example}\label{EB}
For the number of  inversions, we have $\gs=21$ and $b=1$, $\ell_1=2$.
A calculation yields $\gamxx=6$, and \refT{TB} yields
\begin{equation}
  \frac{\ninv(\pinB)-n}{n\qq}\dto N(0,6).
\end{equation}
\end{example}

\begin{remark}
\refT{TB} shows that the typical order of $\ns(\pinB)$ depends 
only on the number of blocks $b$ in $\gs$
(but not on the length $|\gs|$);
more precisely, the asymptotic mean depends only on $b$. 
(Cf.\  the different situation when avoiding $\set{\permD}$ in \refS{SD},
see \refR{RD}.)
Calculations (assisted by Maple) show, however, that the asymptotic
variance $\gamxx$ depends not only on $m$ and $b$; for example 
$\gs=2143=\iotax_2*\iotax_2$ has $\gamxx=6$ while
$\gs=3214=\iotax_3*\iotax_1$ has $\gamxx=52/3$.
\end{remark}

\begin{remark}
The asymptotic variance $\gamxx=0$ when $\gs=\iota_m=1\dotsm m$, in which
case $b=m$ and all blocks have length 1.
This can be seen directly, since all other patterns occur only
$\Op(n^{m-1})$ times (by \refT{TB}), 
and thus $\iota_m$ occurs $\binom nm - \Op(n^{m-1})$ times.
This argument also shows that the asymptotic variance of $n_{1\dotsm m}(\pinB)$ is
of the order $n^{2m-3}$.

It follows from \refP{PUN}
that $\gamxx>0$ for any other
$\gs\in\fSx(\permB)$.
\end{remark}

\section{Avoiding \set{\perm{231}, \perm{321}}}\label{SA}

In this section we consider $T=\set{\permA}$. 
Equivalent sets are 
\set{123, 132},  
\set{123,  213},   
\set{312,  321}.

It was shown by \citet{SS} that $|\fSn(\permA)|=2^{n-1}$, together with the
following characterization (in an equivalent form).

\begin{prop}[{\cite[Proposition 12]{SS}}]\label{PA}
A permutation $\pi$ belongs to the class
$\fSx(\permA)$ if and only if every block in
$\pi$ is of the type $\ell12\dotsm(\ell-1)$ for some $\ell$. 
\qed
\end{prop}

Thus, as in \refS{SB}, a permutation in $\fSx(\permA)$ is determined by its
block lengths, and these can be arbitrary.
In this section, let $\pi_{\ell_1,\dots,\ell_b}$ denote the  permutation in
$\fSx(\permA)$ with block lengths $\ell_1,\dots,\ell_b$.

Again, in an occurrence of $\gs$ in $\pi$, each block in $\gs$ has to be
mapped into a block in $\pi$.
However, this time, several consecutive blocks in $\gs$ may be mapped to the
same block in $\pi$, provided they have length 1.
Moreover, if a block of length $\ell\ge2$ in $\gs$ is mapped to a block in
$\pi$, then the first element has to be mapped to the first element.
Hence, we obtain instead of \eqref{nsB}, if $\gs=\pi_{\ell_1,\dots,\ell_b}$,
\begin{equation}\label{nsA}
  \ns\bigpar{\pi_{L_1,\dots,L_B}} 
= \sum_{1\le i_1<\dots<i_b\le B} \prod_{j=1}^b h_{\ell_i}(L_{i_j})+R,
\end{equation}
where
\begin{equation}\label{hA}
  h_\ell(x):=
  \begin{cases}
    x, & \ell=1,
\\
\binom{x-1}{\ell-1},&\ell\ge2,
  \end{cases}
\end{equation}
and $R$  counts the occurrences where less that $b$ different blocks in
$\pi_{L_1,\dots,L_B}$ are used. 
We represent the block lengths as in \refS{SB}, 
in particular \eqref{Ge}--\eqref{GeL}, again using an
infinite \iid{} sequence $X_i\sim\Ge(\frac12)$.
Then, the main term in \eqref{nsA} is sandwiched between $U$-statistics as
in \eqref{UB}, and we can apply \refP{PUN} to it.
(Alternatively, we can use \refP{PUN2} as in \refR{Ralt}.)

By  \eqref{pgfB}, $\E z^{X-1}=(2-z)\qw$, and calculations similar to 
\eqref{essB} yield
\begin{align}
\E\lrpar{ \binom{X-1}{k}\binom{X-1}{\ell}}
=D(k,\ell),
\qquad k,\ell\ge0.
\end{align}
Hence
\begin{equation}
  \E h_\ell(X)=
  \begin{cases}
    D(\ell-1,0)=1, & \ell\ge2,
\\
2, & \ell=1.
  \end{cases}
\end{equation}
Simple
calculations then yield, 
in addition to \eqref{Balpha}--\eqref{BvarX},
letting $b_1$ be the number of blocks of length 1,
\begin{gather}\label{Aalpha}
\mu
= \prod_{j=1}^b\E h_{\ell_i}(X)
= 2^{b_1},
\\
f_i(X)=
\begin{cases}
2^{b_1}\binom{X-1}{\ell_i-1}, \qquad \ell_i\ge2,
\\  
2^{b_1-1}X,\qquad \ell_i=1, 
\end{cases}
\\
\ggs_{ij}=\Cov\bigpar{f_i(X),f_j(X)}
=
\begin{cases}
  2^{2b_1}D(\ell_i-1,\ell_j-1)-2^{2b_1}, & \ell_i,\ell_j\ge2,
\\
2^{2b_1}\bigpar{\ell_i-1}, & \ell_i\ge2>\ell_j=1,
\\
2^{2b_1-1},& \ell_i=\ell_j=1
\end{cases}
\\
\Cov\bigpar{f_i(X),X}=
\begin{cases}
2^{b_1+1}\bigpar{\ell_i-1}, & \ell_i\ge2,
\\
2^{b_1},& \ell_i=1.
\end{cases}
\label{Aomega}
\end{gather}
Consequently, we obtain by \refPs{PUN} and \ref{PUNmom}
asymptotic normality in the following form.
\begin{theorem}\label{TA}
Let $\gs\in\fSm(\permA)$ have block lengths $\ell_1,\dots,\ell_b$. Then,
as \ntoo,
\begin{equation}\label{ta}
  \frac{\ns(\pinA)-2^{b_1-b}n^b/b!}{n^{b-1/2}}
\dto N\bigpar{0,\gamxx},
\end{equation}
where $\gamxx$ can be calculated by \eqref{cvtau2} and
\eqref{Aalpha}--\eqref{Aomega}.

Moreover, \eqref{ta} holds with convergence of all moments.
\end{theorem}

\begin{proof}
 The argument above yields the stated limit for the first (main) term on the
 \rhs{} of \eqref{nsA}. We show that the remainder term $R$ is negligible.

The term $R$ can be split up as a sum $\sum_{d=1}^{b-1} R_d$, where $R_d$
counts the occurences that use $d$ blocks in $\pi=\pi_{L_1,\dots,L_B}$.
Each $R_d$ may be written as a sum over $d$-tuples of blocks, and thus
bounded as in \eqref{UB} by some \Ustat{s} $U^{(d)}_{\NN+(n)}$
of order $d$.
Applying \refP{PUN} (or \refP{PU}, together with $\NN+(n)\le n$)
to the latter, we find $R_d=\Op(n^d)=\Op(n^{b-1})$, and thus
$R_d/n^{b-1/2}\pto0$.
For moments, we similarly have by \refP{PUNmom} or \ref{PU}
$\E|R_d|^p = O\bigpar{n^{pd}}=O\bigpar{n^{p(b-1)}}=o\bigpar{n^{p(b-1/2)}}$.
Hence, each $R_d$ is negligible in the limit \eqref{ta}, and the result follows.
\end{proof}

\begin{example}\label{EA}
For the number of  inversions, we have $\gs=21$ and $b=1$, $\ell_1=2$, $b_1=0$.
A calculation yields $\gamxx=1/4$, and \refT{TA} yields
\begin{equation}
  \frac{\ninv(\pinA)-n/2}{n\qq}\dto N(0,\tfrac14).
\end{equation}

In fact, we have the exact distribution
\begin{equation}
  \ninv(\pinA)\sim\Bi\bigpar{n-1,\tfrac12}.
\end{equation}
To see this, note that, by \refP{PA},
if we define $\xi_2,\dots,\xi_n$ by
\begin{equation}
\xi_i:=\ett{\text{no block begins at position }i},   
\end{equation}
then every sequence $\xi_2,\dots,\xi_n\in\setoi^{n-1}$ occurs for exactly
one permutation in $\fSn(\permA)$, and thus
$\xi_2,\dots,\xi_n$ are \iid{} $\Be(\frac12)$.
(This is a minor variation of the similar argument in \refS{SB}.)
Furthermore, for each $j\ge2$, the number of inversions $ij$ with $i<j$
equals $\xi_j$, so the total number is $\sum_2^n\xi_i\sim\Bi(n-1,\frac12)$.
\end{example}

\begin{remark}
Unlike in \refS{SB}, here the asymptotic mean depends not only on the
number of blocks in $\gs$, but also on their lengths.
\end{remark}

\begin{remark}
As in \refS{SB},
the asymptotic variance $\gamxx=0$ when $\gs=\iota_m=1\dotsm m$, in which
case $b=m$ and all blocks have length 1,
but
$\gamxx>0$ for any other $\gs\in\fSx(\permA)$.
\end{remark}

\section{Avoiding \set{\perm{132}, \perm{321}}}\label{SE}

In this section we consider $T=\set{\permE}$. 
Equivalent sets are 
\set{123, 231},  
\set{123,  312},   
\set{213,  321}.

It was shown by \citet{SS} that $|\fSn(\permE)|=\binom n2+1$.
(The case $\fSn(\permE)$ is thus more degenerate than the cases considered
above, in the sense that the allowed set of permutations is much smaller;
$|\fSn(\permE)|$ grows polynomially as (roughly) $n^2$, 
compared to $2^{n-1}$ in the previous cases forbidding two permutations of
length 3.)
\cite{SS} gave also the following characterization.
Given $k,\ell\ge1$ and $m\ge0$, let, in this section,  
\begin{equation}\label{pie}
\pi_{k,\ell,m}:=
(\ell+1,\dots,\ell+k,1,\dots,\ell,k+\ell+1,\dots,k+\ell+m)\in\fS_{k+\ell+m}.  
\end{equation}
Thus $\pi_{k,\ell,m}$ 
consists of three 
increasing runs of lengths $k$, $\ell$, $m$
(where the third run is empty when $m=0$).

\begin{prop}[{\cite[Proposition 13]{SS}}]\label{PE}
  \begin{equation}\label{pe}
    \fS_n(\permE)=\bigset{\pi_{k,\ell,n-k-\ell}:k,\ell\ge1,\, k+\ell\le n}
\cup\set{\iota_n}.
  \end{equation}
\qed 
\end{prop}

For asymptotic results, we may ignore the case when $\pinx{\permE}=\iota_n$,
which has probability $1/(\binom n2+1)=o(1)$. 
Conditioning on $\pinx{\permE}\neq\iota_n$, we see by \refP{PE} that
$\pinx{\permE}=\pi_{K,L,n-K-L}$, where $K$ and $L$ are random with
$(K,L)$ uniformly distributed over the set $\set{K,L\ge1:K+L\le n}$.
As \ntoo, we thus have $(K/n,L/n)\dto (X,Y)$ with $(X,Y)$ uniformly
distributed on the triangle $\set{(X,Y)\in\bbR_+^2:X+Y\le1}$. Equivalently,
letting $Z:=1-X-Y$, 
\begin{equation}\label{peq}
\Bigpar{\frac{K}n,\frac{L}n,\frac{n-K-L}{n}}\dto (X,Y,Z)\sim\Dir(1,1,1),
\end{equation}
where we recall that the Dirichlet distribution $\Dir(1,1,1)$ is the uniform
distribution on the simplex $\set{(x,y,z)\in\bbR_+^3: x+y+z=1}$.

If $\gs=\pi_{i,j,p}$ for some $i,j,p$, then it is easily seen that
an  occurrence of $\gs$ in $\pi_{k,\ell,m}$ is obtained by selecting
$i$, $j$ and $p$ elements from the three runs of $\pi_{k,\ell,m}$, and thus  
\begin{equation}\label{pet}
n_{\gs}(\pi_{k,\ell,m})=\binom ki\binom \ell j\binom mp.  
\end{equation}
Similarly, if $\gs=\iota_i$, 
then 
an  occurrence of $\gs$ in $\pi_{k,\ell,m}$ is obtained by selecting $i$
elements from either the union of the first and last run, or from the union
of the two last. Hence, by inclusion-exclusion,
\begin{equation}\label{pett}
n_{\gs}(\iota_i)=\binom {k+m}i+\binom {\ell+m}i-\binom mi.  
\end{equation}

These exact formulas
 together with the description of $\pinx{\permE}$ above and \eqref{peq}
yield  the following asymptotic result.
 
\begin{theorem}\label{TE}
Let $\gs\in\fSx(\permE)$. 
Then   the following hold as \ntoo.
  \begin{romenumerate}
    \item \label{TEa}
If $\gs=\pi_{i,j,p}$ for some $i,j,p$, then 
\begin{equation}\label{tea}
  n^{-(i+j+p)}\ns(\pinx{\permE}) \dto 
W_{i,j,p}:=\frac{1}{i!\,j!\,p!} X^i Y^jZ^p,
\end{equation}
where $(X,Y,Z)\sim \Dir(1,1,1)$.
\item \label{TEb}
If $\gs=\iota_{i}$, then
\begin{equation}\label{teb}
  n^{-i}\ns(\pinx{\permE}) \dto 
W_i:=\frac{1}{i!}\bigpar{(X+Z)^i+(Y+Z)^i-Z^i},
\end{equation}
with $(X,Y,Z)\sim \Dir(1,1,1)$ as in \ref{TEa}.
  \end{romenumerate}
Moreover, these hold jointly
  for any set of such $\gs$, and with convergence of all moments.
In particular, in case \ref{TEa},
\begin{equation}\label{teae}
  n^{-(i+j+p)}\E\ns(\pinx{\permE}) \dto 
\E W_{i,j,p}=\frac{2}{(i+j+p+2)!}
\end{equation}
and in case \ref{TEb},
\begin{equation}\label{tebe}
  n^{-i}\E\ns(\pinx{\permE}) \dto 
\E W_{i}=\frac{4i+2}{(i+2)!}
\end{equation}
\end{theorem}
\begin{proof}
The limits in distribution \eqref{tea} and \eqref{teb} hold 
(with joint convergence)
by the discussion before the theorem. 
Moment convergence holds because the normalized variables in \eqref{tea} and
\eqref{teb} are bounded (by 1).  
Finally, the expectation in \eqref{teae} 
is easily
computed using the multidimensional extension of the beta integral
\cite[(5.14.2)]{NIST}, which implies
\begin{equation}\label{dirabc}
  \E X^a Y^b Z^c = \frac{2\Gamma(a+1)\Gamma(b+1)\gG(c+1)}{\gG(a+b+c+3)},
\qquad a,b,c>-1.
\end{equation}
For the expectation in  \eqref{tebe}, we note also that $X+Z\eqd Y+Z \sim
B(2,1)$; the result follows by a short calculation.
\end{proof}

Higher moments of $W_{i,j,p}$ follow also from \eqref{dirabc}.

\begin{corollary}\label{CE}
The number of inversions
has the asymptotic distribution
\begin{equation}\label{tec}
  n^{-2}\nx{\ww}(\pinx{\permE}) \dto  W:=X Y,
\end{equation}
with $(X,Y)$ as above; the limit variable $W$ has 
density function
\begin{equation}\label{cepdf}
  2\log\bigpar{1+\sqrt{1-4x}} -   2\log\bigpar{1-\sqrt{1-4x}}
,
\qquad 0<x<1/4,
\end{equation}
and moments
\begin{equation}\label{ewr}
  \E W^r = 2\frac{r!^2}{(2r+2)!}, \qquad r>0.
\end{equation}
\end{corollary}

\begin{proof}
 We have $21=\pi_{1,1,0}$, and thus \eqref{tea} yields \eqref{tec}.
The formula \eqref{ewr} for the  moments $\E W^r=\E X^rY^r$ follow by
\eqref{dirabc}.  
Finally, for $0<t<1/4$, $\P(W>t)=\P(XY>t)$ equals 2 times the area of the
set
\set{(x,y)\in\bbR_+^2:x+y\le1,\,xy> t}.
A differentiation and a simple calculation yield \eqref{cepdf}.
\end{proof}

\begin{example}\label{EE}
  For the four allowed patterns of length 3, we find
  \begin{align}
n^{-3}\E n_{123}(\pinx{\permE}) &\to \E W_{3} =\frac{7}{60},
\\   
n^{-3}\E n_{213}(\pinx{\permE}) &\to \E W_{1,1,1} =\frac{1}{60},
\\   
n^{-3}\E n_{231}(\pinx{\permE}) &\to \E W_{2,1,0} =\frac{1}{60},
\\   
\E n_{312}(\pinx{\permE}) &= \E W_{1,2,0} =\frac{1}{60}.
  \end{align}
(See \citet{Zhao} for exact formulas for finite $n$.)
Note that by \eqref{teae}, all $W_{i,j,q}$ with the same $i+j+q$ have the
same expectation; their distributions differ, however, in general, as is
shown by higher moments. For example, in the present example, by
\eqref{dirabc}, 
$\E W_{1,1,1}^2=2/7!$ and $\E W_{2,1,0}^2=3/7!$.
\end{example}

The expected number of occurrences of $\gs$ can also easily be found exactly
for finite $n$, as follows.
As noted above,  \eqref{teae} shows that all $\gs$ in \ref{TEa} of the same
length  
occur in $\pinE$ with asymptotically equal frequencies.
In fact, this holds also exactly, for any $n$.
(Note also that \eqref{teae} is an immediate consequence of \eqref{ten}.)

\begin{theorem}\label{TEN}
  Let $\gs=\pi_{i,j,p}$, with $i,j\ge1$ and $p\ge0$.
Then, for any $n$,
\begin{equation}\label{ten}
  \E n_\gs(\pinE)=\frac{\binom{n+2}{i+j+p+2}}{\binom n2+1}.
\end{equation}
\end{theorem}
\begin{proof}
By \eqref{pet} and the discussion before it, for any given $k,\ell,m$, the
number of occurences of $\gs$ in $\pi_{k,\ell,m}$ equals the number of 
sequences $q_1,\dots,q_i,q'_1,\allowbreak\dots,\allowbreak q'_j,q''_1,\dots,q''_p$ such that
\begin{equation}\label{qqq}
1\le q_1< \dots <q_i \le k < q'_1<\dotsm< q'_j \le\ell<q''_1<\dots<q''_m\le n. 
\end{equation}
  Since $\gs$ does not occur in $\iota_n$, the total number of occurences
of $\gs$ in all elements of $\fSn(\permE)$ is thus, recalling \eqref{pe},
 equal to the number of
all sequences $(q_1, \dots ,q_i , k , q'_1,\dots,q'_j,\ell,q''_1,\dots,q''_m)$
of integers satisfying \eqref{qqq}. By increasing $k$ and all $q'_r$ by 1,
and $\ell$ and all $q''_s$ by 2, we obtain a bijection with the collection
of all subsets of $i+j+q+2$ elements of \set{1,\dots,n+2}.
Hence, the total number of occurrences is $\binom{n+2}{i+j+p+2}$, and
\eqref{ten} follows.
\end{proof}

\section{Avoiding \set{\permAAA}}
\label{S3a}

We proceed to avoiding sets of three permutations.
In this section we avoid $T=\set{\permAAA}$. 
An equivalent set is
\set{123,132,213}.

It was shown by \citet{SS} that $|\fSn(\permAAA)|=F_{n+1}$,
the $(n+1)$th Fibonacci number (with the initial conditions $F_0=0$, $F_1=1$);
they also gave the following characterization (in an equivalent form).

\begin{prop}[{\cite[Proposition $15^*$]{SS}}]\label{PAAA}
A permutation $\pi$ belongs to the class 
$\fSx(\permAAA)$ if and only if every block in
$\pi$ is decreasing and has length $\le 2$, \ie, every block is $1$ or $21$.
\qed
\end{prop}

Cf.\ \refP{PB}; we have here added the restriction that block lengths are 1
or 2.
With this restriction in mind, we use again the notation \eqref{piB} and
note that \eqref{nsB} holds.
A permutation $\pi\in\fSn(\permAAA)$ is thus of the form
$\pi_{L_1,\dots,L_B}$ for some sequence $L_1,\dots,L_B$ of \set{1,2} with
sum $n$; furthermore, this yields a bijection with all such sequences.

Define $p$ to be 
the golden ratio:
\begin{equation}
  p:=
\frac{\sqrt5-1}2,
\end{equation}
so that $p+p^2=1$.
Let $X$ be a random variable with the distribution
\begin{equation}
\P(X=1)=p,\qquad
\P(X=2)=p^2.
\end{equation}
Consider an \iid{} sequence $X_1,X_2,\dots$ of copies of $X$, and let
$S_n:=\sumin X_i$. Then for any sequence $\ell_1,\dots,\ell_b$ with $b\ge1$,
$\ell_i\in\set{1,2}$ and $\sum_{1}^b\ell_i=n$,
\begin{equation}
  \P\bigpar{X_i=\ell_i,\, i=1,\dots,b}
=\prod_{i=1}^b p^{\ell_i}
=p^n.
\end{equation}
This probability is thus the same for all such sequences, which means that,
conditioned on the event that $S_b=n$ for some (unspecified) $b\ge1$, 
the sequence $(X_1,\dots,X_b)$ is equidistributed over all
\set{1,2}-sequences with sum $n$; we have seen above that this equals the
distribution of the sequence of block lengths $(L_1,\dots,L_B)$ of 
a random permutation $\pinAAA$ in $\fSn(\permAAA)$.
Consequently, 
recalling \eqref{NN+} and \refR{RNN},
\begin{equation}
(L_1,\dots,L_B)
\eqd \bigpar{(X_1,\dots,X_{\NN+(n)})\mid S_{\NN+(n)}=n}  .
\end{equation}
It follows from this and \eqref{nsB}
that if $\gs=\pi_{\ell_1,\dots,\ell_b}\in\fSx(\permAAA)$,
and $f$ is defined by \eqref{fB},
then 
$\ns(\pinAAA)$ has the same distribution as 
$U_{\NN+(n)}$ conditioned on $S_{\NN+(n)}=n$. 
Consequently, \refP{PUN2} applies and yields asymptotic normality of
$\ns(\pinAAA)$, and \refP{PUNmom} adds moment convergence.

To find the parameters, let $\gs$ have $b_1$ blocks of length 1 and $b_2$
blocks of length 2 (so $b_1+b_2=b$ and $b_1+2b_2=|\gs|$).
Then, noting $\binom X2=X-1$,
\begin{align}\label{AAAalpha}
  \nu&=\E X=p+2p^2=2-p=\frac{5-\sqrt5}{2},
\\
\Var X&=p^3=2p-1=\sqrt5-2,
\\
\E\binom{X}2&=\P(X=2)=p^2=1-p,
\\
\mu&=  (2-p)^{b_1}(1-p)^{b_2} 
= \Bigparfrac{5-\sqrt5}{2}^{b_1}\Bigparfrac{3-\sqrt5}{2}^{b_2},
\label{muAAA}
\\
f_i(X)&=
\begin{cases}
  (2-p)^{b_1-1}(1-p)^{b_2} X, & \ell_i=1,
\\
  (2-p)^{b_1}(1-p)^{b_2-1}(X-1), & \ell_i=2,
\end{cases}
\\
\ggs_{ij}&=
\begin{cases}
 (2-p)^{2b_1-2}(1-p)^{2b_2}(2p-1),  & \ell_i=\ell_j=1,\\
 (2-p)^{2b_1-1}(1-p)^{2b_2-1}(2p-1),  & \ell_i=1< \ell_j=2,\\
 (2-p)^{2b_1}(1-p)^{2b_2-2}(2p-1),  & \ell_i=\ell_j=2.
\end{cases}
\\
\Cov&\bigpar{f_i(X),X}=
\begin{cases}
 (2-p)^{b_1-1}(1-p)^{b_2}(2p-1),  & \ell_i=1,
\\
 (2-p)^{b_1}(1-p)^{b_2-1}(2p-1),  & \ell_i=2.
\end{cases}
\label{AAAomega}
\end{align}

We summarize.
\begin{theorem}
  \label{TAAA}
Let $\gs\in\fSm(\permAAA)$ have block lengths $\ell_1,\dots,\ell_b$. Then,
as \ntoo,
\begin{equation}\label{taz}
  \frac{\ns(\pinAAA)-\mu n^b/b!}{n^{b-1/2}}
\dto N\bigpar{0,\gamxx},
\end{equation}
where $\mu$ is given by \eqref{muAAA} and
$\gamxx$ can be calculated by \eqref{cvtau2} and
\eqref{AAAalpha}--\eqref{AAAomega}.

Moreover, \eqref{taz} holds with convergence of all moments.
\qed
\end{theorem}

\begin{example}
  \label{EAAA}
For the number of inversions, $\gs=21$, $b=1=b_2$ and $b_1=0$.
Hence, $\mu=1-p=(3-\sqrt5)/2$ and, 
by a calculation, $\gamxx=(2-p)^{-3}\Var X= 5^{-3/2}$.
Consequently,
\begin{equation}\label{tazi}
  \frac{\ninv(\pinAAA)-\frac{3-\sqrt5}{2} n}{n^{1/2}}
\dto N\bigpar{0,5^{-3/2}}.
\end{equation}
\end{example}

\begin{remark}
  Again, $\gamxx>0$ unless $\gs=\iota_m$.
\end{remark}

\section{Avoiding \set{\permCCC}} 
\label{S3c}

In this section we avoid $\set{\permCCC}$.
Equivalent sets are
\set{132,213,231},
\set{132,213,312},
\set{213,231,312}.

It was shown by \citet{SS} that $|\fSn(\permCCC)|=n$,
together with the following characterization (in an equivalent form).
In this section, let
\begin{equation}\label{piCCC}
  \pi_{k,\ell}:=\iotax_k*\iota_{l}=(k,\dots,1,k+1,\dots,k+\ell)
\in\fS_{k+\ell},
\qquad k\ge1,\,\ell\ge0.
\end{equation}
Note that $\pi_{1,\ell}=\iota_{1+\ell}$.

\begin{prop}[{\cite[Proposition $16^*$]{SS}}]\label{PCCC}
  \begin{equation*}
 \fSn(\permCCC)=\set{\pi_{k,n-k}:1\le k\le n}.
\qedtag
  \end{equation*}
\end{prop}
Cf.\ \refPs{PD} and \ref{PB}, which characterize supersets.
(Equivalently, $\pi\in\fSx(\permCCC)$ if the first block is decreasing and
all other blocks have length 1.)

Hence, the random $\pinCCC=\pi_{K,n-K}$, where $K\in[n]$ is uniformly
random.
Obviously, as \ntoo,
\begin{equation}\label{limCCC}
  K/n\dto U\sim \Uoi.
\end{equation}
Furthermore, if $\gs=\pi_{k,\ell}$, then it is easy to see that
\begin{equation}\label{nsCCC}
  \ns\bigpar{\pi_{K,n-K}}=
  \begin{cases}
\binom{K}{k}\binom{n-K}{\ell}, & k\ge2,
\\[5pt]
K\binom{n-K}{\ell}+\binom{n-K}{\ell+1}, & k=1.
  \end{cases}
\end{equation}

\begin{theorem}\label{TCCC}
Let $\gs\in\fSx(\permCCC)$. 
Then   the following hold as \ntoo,
with $U\sim\Uoi$.
  \begin{romenumerate}
    \item \label{TCCCa}
If $\gs=\pi_{k,m-k}$ with $2\le k\le m$, then
\begin{equation}\label{tccca}
  n^{-m}\ns(\pinCCC) \dto 
W_{k,m-k}:=\frac{1}{k!\,(m-k)!} U^k (1-U)^{m-k}.
\end{equation}
\item \label{TCCCb}
If $\gs=\pi_{1,m-1}=\iota_{m}$, then
\begin{equation}\label{tcccb}
  \begin{split}
  n^{-m}\ns(\pinx{\permCCC}) \dto 
W_{1,m-1}&:=\frac{1}{(m-1)!}U(1-U)^{m-1}+\frac{1}{m!}(1-U)^m
\\&\phantom:
=\frac{1}{m!}\bigpar{1+(m-1)U}(1-U)^{m-1}.
  \end{split}
\end{equation}
  \end{romenumerate}
Moreover, these hold jointly
  for any set of such $\gs$, and with convergence of all moments.
In particular, in case \ref{TCCCa},
\begin{equation}\label{tcccae}
  n^{-m}\E\ns(\pinx{\permCCC}) \dto 
\E W_{k,m-k}=\frac{1}{(m+1)!},
\qquad k\ge2,
\end{equation}
and in case \ref{TCCCb},
\begin{equation}\label{tcccbe}
  n^{-m}\E\ns(\pinx{\permCCC}) \dto 
\E W_{1,m-1}=\frac{2}{(m+1)!}.
\end{equation}
\end{theorem}
\begin{proof}
The limits in distribution \eqref{tccca} and \eqref{tcccb} hold 
(with joint convergence)
by \eqref{nsCCC} and \eqref{limCCC}.
Moment convergence holds because the normalized variables in \eqref{tccca} and
\eqref{tcccb} are bounded (by 1).  
Finally, the expectations in \eqref{tcccae}--\eqref{tcccbe} 
are computed by standard beta integrals.
\end{proof}

\begin{corollary}\label{CCCC}
The number of inversions
has the asymptotic distribution
\begin{equation}\label{cccc}
  n^{-2}\nx{\ww}(\pinx{\permCCC}) \dto W:=U^2/2
\end{equation}
with $U\sim\Uoi$.
Thus, $2W\sim B(\frac12,1)$, and $W$ has moments
\begin{equation}\label{ewr}
  \E W^r = \frac{1}{2^r(2r+1)}, \qquad r>0.
\end{equation}
\end{corollary}

\begin{proof}
 We have $21=\pi_{2,0}$ by \eqref{piCCC}, and \eqref{tccca} yields \eqref{cccc}.
The remaining statements follow by simple calculations.
\end{proof}

\section{Avoiding \set{\permBBB}} 
\label{S3b}

In this section we avoid $\set{\permBBB}$.
Equivalent sets are
\set{123,132,231},
\set{123,213,312},
\set{213,312,321},
\set{123,132,312},
\set{123,213,231},
\set{132,312,321},
\set{213,231,321}.

It was shown by \citet{SS} that $|\fSn(\permBBB)|=n$,
together with the following characterization (in an equivalent form).
In this section, let
\begin{equation}\label{piBBB}
  \pi_{k,\ell}:=(k,1,\dots,k-1,k+1,\dots,k+\ell)\in\fS_{k+\ell},
\qquad k\ge1,\,\ell\ge0.
\end{equation}
Note that $\pi_{k,\ell}$ equals $\pi_{1,k-1,\ell}$ in the notation \eqref{pie} of
\refS{SE} if $k\ge2$, and $\iota_{1+\ell}$ if $k=1$.

\begin{prop}[{\cite[Proposition $16^*$]{SS}}]\label{PBBB}
  \begin{equation*}
 \fSn(\permBBB)=\set{\pi_{k,n-k}:1\le k\le n}.
\qedtag
  \end{equation*}
\end{prop}
Cf.\ \refP{PE}, which characterizes a superset.

Hence, the random $\pinBBB=\pi_{K,n-K}$, where $K\in[n]$ is uniformly
random.
Obviously, as \ntoo, \eqref{limCCC} holds in this case too.
Furthermore, if $\gs=\pi_{k,\ell}$, then it is easy to see,
\eg{} by \eqref{pet}--\eqref{pett}, that
\begin{equation}\label{nsBBB}
  \ns\bigpar{\pi_{K,n-K}}=
  \begin{cases}
\binom{K-1}{k-1}\binom{n-K}{\ell}, & k\ge2,
\\[5pt]
\binom{n-1}{\ell+1}+\binom{n-K}{\ell}, & k=1.
  \end{cases}
\end{equation}

\begin{theorem}\label{TBBB}
Let $\gs\in\fSx(\permBBB)$. 
Then   the following hold as \ntoo,
with $U\sim\Uoi$.
  \begin{romenumerate}
    \item \label{TBBBa}
If $\gs=\pi_{k,m-k}$ with $2\le k\le m$, then
\begin{equation}\label{tbbba}
  n^{-(m-1)}\ns(\pinBBB) \dto 
W_{k,m-k}:=\frac{1}{(k-1)!\,(m-k)!} U^{k-1} (1-U)^{m-k}.
\end{equation}
\item \label{TBBBb}
If $\gs=\pi_{1,m-1}=\iota_{m}$, then
\begin{equation}\label{tbbbb}
  \begin{split}
  n^{-m}\ns(\pinx{\permBBB}) =\frac{1}{m!}+O\bigpar{n\qw} \pto \frac{1}{m!}.
  \end{split}
\end{equation}
  \end{romenumerate}
Moreover, these hold jointly
  for any set of such $\gs$, and with convergence of all moments.
In particular, in case \ref{TBBBa},
\begin{equation}\label{tbbbae}
  n^{-(m-1)}\E\ns(\pinx{\permBBB}) \dto 
\E W_{k,m-k}=\frac{1}{m!},
\qquad k\ge2.
\end{equation}
\end{theorem}
\begin{proof}
By \eqref{nsBBB} and \eqref{limCCC},
similarly to the proof of \refT{TCCC}.
\end{proof}

\begin{corollary}\label{CBBB}
The number of inversions
$\ninv(\pinBBB)$ has a uniform distribution on \set{0,\dots,n-1}, and thus  
the asymptotic distribution
\begin{equation}\label{bbbc}
  n^{-1}\ninv(\pinx{\permBBB}) \dto U\sim\Uoi.
\end{equation}
\end{corollary}
\begin{proof}
By \eqref{piBBB}, $12=\pi_{2,0}$, and thus 
\eqref{nsBBB} yields $\ninv(\pi_{K,n-K})=K-1$.
\end{proof}

\section{Avoiding \set{\permEEE}}  
\label{S3e} 
In this section we avoid $\set{\permEEE}$.
An
equivalent sets is
\set{123,231,312}.

It was shown by \citet{SS} that $|\fSn(\permEEE)|=n$,
together with the following characterization (in an equivalent form).
In this section, let
\begin{equation}\label{piEEE}
  \pi_{k,\ell}:=(\ell+1,\dots,\ell+k,1,\dots,\ell)\in\fS_{k+\ell},
\qquad k\ge1,\,\ell\ge0.
\end{equation}
Note that $\pi_{k,\ell}$ equals $\pi_{k,\ell,0}$ in the notation \eqref{pie} of
\refS{SE} if $\ell\ge1$, and $\iota_{k}$ if $\ell=0$.

\begin{prop}[{\cite[Proposition $16^*$]{SS}}]\label{PEEE}
  \begin{equation*}
 \fSn(\permEEE)=\set{\pi_{k,n-k}:1\le k\le n}.
\qedtag
  \end{equation*}
\end{prop}
Cf.\ \refP{PE}, which again characterizes a superset.

Hence, the random $\pinEEE=\pi_{K,n-K}$, where $K\in[n]$ is uniformly
random, and \eqref{limCCC} holds again.
Furthermore, if $\gs=\pi_{k,\ell}$, then it is easy to see,
\eg{} by \eqref{pet}--\eqref{pett}, that
\begin{equation}\label{nsEEE}
  \ns\bigpar{\pi_{K,n-K}}=
  \begin{cases}
\binom{K}{k}\binom{n-K}{\ell}, & \ell\ge1,
\\[5pt]
\binom{K}{k}+\binom{n-K}{k}, & \ell=0.
  \end{cases}
\end{equation}

\begin{theorem}\label{TEEE}
Let $\gs\in\fSx(\permEEE)$. 
Then   the following hold as \ntoo,
with $U\sim\Uoi$.
  \begin{romenumerate}
    \item \label{TEEEa}
If $\gs=\pi_{k,m-k}$ with $1\le k\le m-1$, then
\begin{equation}\label{teeea}
  n^{-m}\ns(\pinEEE) \dto 
W_{k,m-k}:=\frac{1}{k!\,(m-k)!} U^{k} (1-U)^{m-k}.
\end{equation}
\item \label{TEEEb}
If $\gs=\pi_{m,0}=\iota_{m}$, then
\begin{equation}\label{teeeb}
  n^{-m}\ns(\pinEEE) \dto 
W_{m,0}:=\frac{1}{m!} \bigpar{U^{m}+ (1-U)^{m}}.
\end{equation}
  \end{romenumerate}
Moreover, these hold jointly
  for any set of such $\gs$, and with convergence of all moments.
In particular, in case \ref{TEEEa},
\begin{equation}\label{teeeae}
  n^{-m}\E\ns(\pinx{\permEEE}) \dto 
\E W_{k,m-k}=\frac{1}{(m+1)!},
\qquad 1\le k < m,
\end{equation}
and in case \ref{TEEEb},
\begin{equation}\label{teeebe}
  n^{-m}\E\ns(\pinx{\permEEE}) \dto 
\E W_{m,0}=\frac{2}{(m+1)!}.
\end{equation}
\end{theorem}
\begin{proof}
By \eqref{nsEEE} and \eqref{limCCC},
similarly to the proof of \refT{TCCC}.
\end{proof}

\begin{corollary}\label{CEEE}
The number of inversions
has the asymptotic distribution
\begin{equation}\label{ceee}
  n^{-2}\nx{\ww}(\pinx{\permEEE}) \dto W:=U(1-U),
\end{equation}
with $U\sim\Uoi$.
Thus, $4W\sim B(1,\frac12)$, and $W$ has moments
\begin{equation}\label{ewr}
  \E W^r = \frac{\gG(r+1)^2}{\gG(2r+2)}, \qquad r>0.
\end{equation}
\end{corollary}

\begin{proof}
 We have $21=\pi_{1,1}$ by \eqref{piEEE}, 
and thus \eqref{teeea} yields \eqref{ceee}.
The remaining statements follow by simple calculations, using $4W=1-(2U-1)^2$
and a beta integral.
\end{proof}

\newcommand\AMS{Amer. Math. Soc.}
\newcommand\Springer{Springer-Verlag}
\newcommand\Wiley{Wiley}

\newcommand\vol{\textbf}
\newcommand\jour{\emph}
\newcommand\book{\emph}
\newcommand\inbook{\emph}
\def\no#1#2,{\unskip#2, no. #1,} 
\newcommand\toappear{\unskip, to appear}

\newcommand\arxiv[1]{\url{arXiv:#1.}}
\newcommand\arXiv{\arxiv}

\def\nobibitem#1\par{}

\end{document}